\documentclass[11pt]{article}

\usepackage{amssymb}
\usepackage{amsmath}

\usepackage{tikz}

\usepackage{ulem}

\usepackage[colorlinks=true, breaklinks = true, linkcolor=blue, citecolor=blue, urlcolor=blue]{hyperref}

\usepackage{multicol}

\usepackage{amsthm}

\newtheorem{property}{Property}

\newtheorem{theorem}{Theorem}
\newtheorem{corollary}{Corollary}

\newtheorem{observation}{Observation}

\begin{document}

\title{Robust single-stage selection problems\\ with budgeted interval uncertainty}

\author{Antoine Lhomme$^1$ \and Nadia Brauner$^1$ \and Evgeny Gurevsky$^2$
\and Mikhail Y. Kovalyov$^3$ \and Erwin Pesch$^4$\\[0.5cm]
{\normalsize $^1$ Université Grenoble Alpes, CNRS, Grenoble INP, G-SCOP, France}\\
\normalsize $^2$ LS2N, University of Nantes, France \\
\normalsize $^3$ National Academy of Sciences of Belarus, Minsk, Belarus\\
\normalsize $^4$ University of Siegen, Germany,  HHL Leipzig Graduate School of Management, Germany
}

\maketitle

\begin{abstract} 
We study single-stage decision problems in which a subset of items with minimum total cost has to be selected at once from a given set of items, subject to two costs of each item -- fixed and uncertain -- and cardinality constraints for each cost type. The worst-case budgeted interval uncertainty is considered. At the time of decision making, the fixed costs are known, but for each uncertain cost, only the range of its values is available. Similar but two-stage selection problems have been studied in the literature, in which first- and second-stage decisions are made before and after uncertain costs become known, respectively.  The problems studied are distinguished by continuous or discrete uncertain costs, and by uncertainty budgets based on cardinality or volume. An almost complete computational complexity classification is provided, including fast polynomial-time algorithms, NP- and $\Sigma^p_2$-completeness and hardness proofs. 
\end{abstract}

{\bf keyword}
robust optimization --  budgeted uncertainty -- selection problem -- dynamic programming -- computational complexity





\section{Introduction and literature review} \label{SecIntrod} 
Kasperski and Zieli\'nski \cite{Kasperski2017} and Chassein et al. \cite{Chassein2018} introduce several robust cost minimization item selection problems with interval uncertainty of the costs. They note that in practice these problems often have a two-stage nature such that the item costs in the first stage are precisely known, while the item costs in the second stage are uncertain.  The worst-case realization of uncertain costs is considered, which can be viewed as the policy of an {\it adversary}. 

In \cite{Kasperski2017} and \cite{Chassein2018}, it is assumed that the selection decision is made in two stages, before and after the uncertainty is realized. We study single-stage problems, where the selection decision has to be made once before the uncertainty is realized. The single-stage problem arises in situations where the selected (manufactured, purchased, completed) items (products, projects) are required to be specified in a long-term contract. The contract should specify which fixed-cost items will be delivered first, and which uncertain-cost items will be delivered later. Optimal decisions for the single-stage cost minimization problems provide upper bounds on the minima for their two-stage counterparts in \cite{Kasperski2017} and \cite{Chassein2018}, since they use less information than the two-stage ones.

In the deterministic version of the item selection problem, there is a set of 
 $n$ items with given costs, and the objective is to select $p$ items such that their total cost is minimized. An $O(n)$ time algorithm for this problem is based on the median finding technique of Blum et~al.~\cite{Blum1973}: find the $p$-th smallest item cost $c$ and then select $p$ items with costs no greater than $c$. The problem becomes more difficult if the item costs are uncertain. 

Chassein et al. \cite{Chassein2018} consider the case of \textit{budgeted uncertainty}, the concept of which was introduced by Bertsimas and Sim \cite{Bertsimas2003, Bertsimas2004}. According to this concept, the vector of item costs can be any vector from a set of \textit{discrete} or \textit{continuous interval scenarios}. The difference between the two types of scenario is that the uncertain cost can take one of two values -- the smallest or the largest -- in the discrete scenario, and can take any value between the smallest and the largest in the continuous scenario. The budgeted uncertainty sets include item cost vectors which are  within a given range, denoted as $\Gamma$ and called \textit{uncertainty budget}, from a given cost vector. 

Depending on the cardinality of the item sets selected in the first and second stages, two cases are studied in \cite{Chassein2018}. They  are denoted as problems {\sc RREC} (Robust Recoverable) and {\sc R2ST} (Robust Two-Stage). In {\sc RREC}, $p$ items have to be selected in both stages and no item can be selected twice, while in {\sc R2ST}, $p$ items have to be selected in each stage and at least $p-k$ items have to be selected twice (no more than $k$ new items have to be selected in the second stage). 

The polynomial solvability of the continuous variants of {\sc RREC} and {\sc R2ST} is proven in~\cite{Chassein2018} by analyzing their mixed-integer linear programming formulations. Discrete variants of these problems are proven to be NP-hard in the ordinary sense by Goerigk et~al.~\cite{Goerigk2022}. The questions of whether the discrete variants are pseudo-polynomially solvable or NP-hard in the strong sense are yet to be answered, as well as their hardness in some class of the polynomial-time hierarchy (see, {\it e.g.}, Stockmeyer \cite{Stockmeyer1976}), for example, $\Sigma^p_2$-hardness.  Kasperski and Zieli\'nski \cite{Kasperski2017} propose $O(n)$ and $O(n^2(p-k+1))$ algorithms for the continuous variants of {\sc RREC} and {\sc R2ST}, respectively, with the non-restricting uncertainty budget. 

In the next section, we formulate single-stage counterparts of {\sc RREC} and {\sc R2ST} with the budgeted interval uncertainty. The uncertainty budget can be of two types: volume budget accounting for the total increase of the uncertain costs (with respect to their smallest nominal values), and cardinality budget accounting for the number of items with non-zero increase of the uncertain cost. 
Properties of the discrete single-stage problems with cardinality budget are established in Section \ref{SecDP}, and $O(np^2)$ and $O(nk^2p)$ time dynamic programming algorithms for the  {\sc RREC} and {\sc R2ST} counterparts, respectively, are developed in this case. Note that the same two-stage problems are NP-hard due to Goerigk et~al.~\cite{Goerigk2022}. 

The polynomial solvability of models with budgeted uncertainty was the main motivation for the introduction of this type of uncertainty by Bertsimas and Sim \cite{Bertsimas2003, Bertsimas2004}.  The continuous single-stage counterparts of the problems  {\sc RREC} and {\sc R2ST} with cardinality budget are  equivalent to their discrete analogs, and therefore, they can be solved by the same polynomial-time algorithms. 

In Section~\ref{SecU_D1}, we consider volume budget and prove that the single-stage problems are NP-hard in the ordinary sense and $\Sigma_2^p$-hard for discrete variants and that they are polynomially solvable for continuous variants.  Note that a $\Sigma^p_2$-hard problem cannot be formulated in polynomial time as a mixed-integer linear program (MILP) of polynomial size if the commonly accepted conjecture holds (see, {\it e.g.}, Woeginger \cite{Woeginger2021}). 

Some of the obtained results are adjusted for the more general {\it ``weighted uncertainty''} case in Section \ref{SecW}. In this case, weights are associated with the items, and contributions of the uncertain costs to the uncertainty budget are weighted.  The paper concludes  with a table of complexity results for the studied problems and suggestions for future research. 

\section{Formulating single-stage problems} \label{SForm} The problems in this paper are to make a single-stage minimum cost selection of items from a given set. Each item is associated with a fixed cost and an uncertain cost. The decision consists of two parts: selecting items with fixed costs and selecting items with uncertain costs.  There are cases in which the same item can be selected twice -- with both fixed and uncertain cost. We call the total uncertain cost the {\it adversarial cost}.

Single-stage item selection problems studied in this paper can be formulated as follows. Denote $N=\{1,\ldots,n\}$. There are given non-negative rational vectors $C=(C_1,\ldots,C_n)$, $\underline c=(\underline c_1,\ldots,\underline c_n)$ and $d=(d_1,\ldots,d_n)$, a positive rational number~$\Gamma$, variable binary vectors $x=(x_1,\ldots,x_n)$ and $y=(y_1,\ldots,y_n)$, a set $X\times Y$ of feasible pairs $(x,y)$, and a variable non-negative real vector $\delta=(\delta_1,\ldots,\delta_n)$. Below we formulate in a vector form continuous and discrete single-stage item selection problems with two types of interval uncertainty budget, which takes into account the volume or the number of uncertain costs deviating from their nominal (minimum) values. We denote these problems as {\sc Con-Vol}, {\sc Dis-Vol}, {\sc Con-Car} and {\sc Dis-Car}. 

{\bf Problem}  {\sc Con-Vol} (continuous uncertainty with volume budget):
$$\min_{(x,y)\in X\times Y}\left\{Cx+\max\limits_{0\le \delta\le d}\Big\{(\underline c+\delta)y: \sum_{i\in N}\delta_i\le\Gamma\Big\}\right\}.$$

{\bf Problem}  {\sc Dis-Vol} (discrete uncertainty with volume budget):$$\min_{(x,y)\in X\times Y}\left\{Cx+\max\limits_{\delta_i\in \{0,d_i\},i\in N}\Big\{(\underline c+\delta)y: \sum_{i\in N}\delta_i\le\Gamma\Big\}\right\}.$$

{\bf Problem}  {\sc Con-Car} (continuous uncertainty with cardinality budget):$$\min_{(x,y)\in X\times Y}\left\{Cx+\max\limits_{0\le \delta\le d}\Big\{(\underline c+\delta)y: |\{i\in N:\delta_i>0\}|\le\Gamma\Big\}\right\}.$$

{\bf Problem} {\sc Dis-Car} (discrete uncertainty with cardinality budget):$$\min_{(x,y)\in X\times Y}\left\{Cx+\max\limits_{\delta_i\in \{0,d_i\},i\in N}\Big\{(\underline c+\delta)y: |\{i\in N:\delta_i=d_i\}|\le\Gamma\Big\}\right\}.$$

To simplify notation and facilitate further discussion, we denote solutions of these problems as ordered triples $(x,y,\delta)$, although only $x$ and $y$ represent the solution of the decision maker and $\delta$ is the strategy of the adversary. 
Let $p$ and $k$ be given positive integer numbers such that $0 \le  k\le p\le n$. In this paper, we study the following special cases of  these problems:
\begin{itemize}
\item {\sc Con-Vol}($p$), {\sc Dis-Vol}($p$), {\sc Con-Car}($p$), {\sc Dis-Car}($p$) -- special cases of the problems {\sc Con-Vol}, {\sc Dis-Vol}, {\sc Con-Car} and {\sc Dis-Car}, respectively, in which 
$$X\times Y=\Big\{(x,y): \sum_{i\in N}(x_i+y_i)=p, \sum_{i\in N} x_iy_i=0\Big\}.$$ 
These single-stage problems are counterparts of the two-stage problem {\sc RREC} in \cite{Chassein2018}. 
\item  {\sc Con-Vol}($p,k$), {\sc Dis-Vol}($p,k$), {\sc Con-Car}($p,k$), {\sc Dis-Car}($p,k$) -- special cases of the problems {\sc Con-Vol}, {\sc Dis-Vol}, {\sc Con-Car} and {\sc Dis-Car}, respectively, in which $$X\times Y=\Big\{(x,y): \sum_{i\in N} x_i=p,  \sum_{i\in N} y_i=p, \sum_{i\in N} x_iy_i\ge p-k\Big\}.$$  These single-stage problems are counterparts of the two-stage problem {\sc R2ST} in \cite{Chassein2018}. 
\end{itemize}

The uncertain cost of each item $i$ is of the form $\underline c_i+\delta_i$. In the continuous problems, $\delta_i$ can take any real value from $[0,d_i]$, and in the discrete problems $\delta_i\in\{0,d_i\}$. In the problems {\sc Con-Vol}($p$), {\sc Dis-Vol}($p$), {\sc Con-Car}($p$) and {\sc Dis-Car}($p$), at most $p$ items with fixed costs have to be selected and, if only $r$ (where $r\le p$) such items have been selected, then $p-r$ different items with uncertain costs should be selected. In the problems  {\sc Con-Vol}($p,k$), {\sc Dis-Vol}($p,k$), {\sc Con-Car}($p,k$) and {\sc Dis-Car}($p,k$), exactly~$p$ items with fixed costs and exactly $p$ items with uncertain costs have to be selected, and at most~$k$ (where  $0\le k\le p$) selected items with uncertain costs can be different from the selected items with fixed costs. 

Vector $x$ represents the selection of items with fixed costs, and vector $y$ represents the selection of items with uncertain costs. The uncertainty is assumed to have a budget expressed by $\Gamma$. In the problems {\sc Con-Vol}($p$), {\sc Dis-Vol}($p$), {\sc Con-Vol}($p,k$) and {\sc Dis-Vol}($p,k$) with {\it volume budget}, the budget constraint is represented by the relation $\sum_{i\in N}\delta_i\le\Gamma$. In the problems {\sc Con-Car}($p$), {\sc Dis-Car}($p$), {\sc Con-Car}($p,k$) and {\sc Dis-Car}($p,k$) with {\it cardinality budget}, it is driven by the relation $|\{i\in N:\delta_i>0\}|\le\Gamma$.

\section{Cardinality budget: polynomial algorithms} \label{SecDP} We begin by considering the discrete problems with cardinality budget.
Suppose the items are re-numbered such that \mbox{$d_1\ge \cdots\ge d_n$.} In this section, assume without loss of generality that $\Gamma$ is integer. If it is non-integer, re-set $\Gamma:=\lfloor\Gamma\rfloor$.
Our algorithm for the  problem {\sc Dis-Car}($p$) is based on the following property.  

\begin{property} \label{PRUDr2} There exists an optimal solution $(x,y,\delta)$ of the problem {\sc Dis-Car}$(p)$ such that \mbox{$\delta_i=d_i$} for $\min\{\sum_{j\in N} y_j,\Gamma\}$ smallest indices $i$ among those with $y_i=1$. The remaining $\delta_i$ values are equal to zero. 
\end{property} 

\begin{proof} Consider an optimal solution $(x,y,\delta)$ of the problem {\sc Dis-Car}($p$) and assume that Property 
\ref{PRUDr2} is not satisfied. Then, there exist two indices $i$ and $j$ such that $i<j$, $y_i=y_j=1$, $\delta_i=0$ and $\delta_j=d_j$. Modify this solution by re-setting $\delta_i=d_i$ and $\delta_j=0$. This modification does not change the adversarial cost $\max\limits_{\delta_i\in \{0,d_i\},i\in N}\Big\{\delta y: |\{i\in N:\delta_i=d_i\}|\le\Gamma\Big\}$ if $d_i=d_j$ 
and it increases this cost if $d_i>d_j$. Therefore, the new solution remains optimal for the adversary in the first case and the original solution was not optimal for the adversary in the second case. Repetition of this modification a finite number of times completes the proof.
\end{proof}

For a given $y$, introduce index $r(y)=\max\{j: y_j=1,j\in N\}$ and set 
$$I(y)=\Big\{j: y_j=1, \sum_{h=1}^j y_h\le\Gamma,1\le j\le r(y)\Big\}.$$ 
Due to Property \ref{PRUDr2}, the problem  {\sc Dis-Car}($p$) reduces to the problem 
$$\min_{(x,y)}\Big\{Cx+\underline cy+\sum_{i\in I(y)}d_iy_i: \sum_{i\in N}(x_i+y_i)=p, \sum_{i\in N} x_iy_i=0\Big\},$$
for which we keep the same notation  {\sc Dis-Car}($p$). For this problem, the $\delta$-part of the solution is fully determined by its $y$-part. This problem can be solved by the following dynamic programming algorithm. 

The algorithm iteratively constructs partial solutions $(x,y)$, in which variables $x_i$ and $y_i$ are determined for $i=1,\ldots,j$, and $j$ is the size of vectors $x$ and $y$ in iteration $j$, $j\in N$. Recall that $d_1\ge \cdots\ge d_n$. Each partial solution $(x,y)$ is associated with a state $(j,k_x,k_y)$, where 
 $k_x=|\{i:x_i=1,y_i=0\}|$ and $k_y=|\{i:x_i=0,y_i=1\}|$. For each state $(j,k_x,k_y)$, the function $G_j(k_x,k_y)$ is recursively calculated, which is the minimum value of $Cx+\underline cy+\sum_{i\in I(y)}d_iy_i$ among all partial solutions $(x,y)$ in this state. It is clear that if a partial solution in the state $(j,k_x,k_y)$ can be extended to a complete optimal solution of the problem  {\sc Dis-Car}($p$), then a solution with the value $G_j(k_x,k_y)$ can be extended in the same way to an optimal solution as well. 

The initialization is $G_0(0,0)=0$ and the recursion for $j\in N$, $k_x=0,1,\ldots,p$,  $k_y=0,1,\ldots,p$, $k_x+k_y\le p$, is
\begin{align*}
G_j(k_x,k_y)=\min\begin{cases} G_{j-1}(k_x,k_y), {\rm\ if\ }  k_x+k_y<j,                               & [x_j=0, y_j=0]  \cr
G_{j-1}(k_x-1,k_y)+C_j,                                                         & [x_j=1, y_j=0]  \cr
G_{j-1}(k_x,k_y-1)+\underline c_j+d_j, {\rm\ if\ }  k_y\le \Gamma,               & [x_j=0, y_j=1]  \cr
G_{j-1}(k_x,k_y-1)+\underline c_j, {\rm\ if\ } k_y > \Gamma.                 & [x_j=0, y_j=1]
\end{cases}
\end{align*}

Thus, the optimal solution value is equal to $G^*=\min\{G_n(k_x,k_y): k_x+k_y=p\}$.
The corresponding optimal solution $(x^*,y^*)$ can be found by going backwards through the recursive equations. The time and space requirements of this algorithm are equal to the cardinality of the state space, which is $O(np^2)$.

The following property is the basis of our algorithm for the  problem  {\sc Dis-Car}($p,k$).

\begin{property} \label{PRUD2rr}  
Let $(x,y)$ be an optimal solution  of the problem  {\sc Dis-Car}$(p,k)$. Then, solution $(x,y,\delta^{(y)})$ is also optimal, where $\delta^{(y)}_i=d_i$ for $i\in I(y)$ and $\delta^{(y)}_i=0$ for $i\not\in I(y)$.

\end{property}

\begin{proof} 
Consider the case $\Gamma\ge p$. Then, for any feasible solution $y$, we have $|\{i: y_i=1,i\in N\}|=p\le\Gamma$ and $\{i: y_i=1,i\in N\}=I(y)$. Hence, the value of $\max\limits_{\delta_i\in \{0,d_i\},i\in N}\{\delta y: |\{i\in N:\delta_i=d_i\}|\le\Gamma\}$ is achieved at $\delta_i=d_i$ for $y_i=1$, $i\in N$, and, as a consequence, for $i\in I(y)$. Now, consider the case $\Gamma< p$. In this case, the maximum value above is evidently achieved for $\delta_i=d_i$, $i\in I(y)$, and $\delta_i=0$ for $i\not\in I(y)$.
\end{proof}

Due to Property \ref{PRUD2rr}, in the case $\Gamma\ge p$, equality $I(y)=\{i: y_i=1, i\in N\}$ is satisfied, and the problem {\sc Dis-Car}($p,k$) reduces to the problem 
\begin{equation} \label{ProbGam} \min_{(x,y)}\Big\{Cx+(\underline c+d)y: \sum_{i\in N} x_i=p,  \sum_{i\in N} y_i=p, \sum_{i\in N} x_iy_i\ge p-k\Big\}.\end{equation} 
In the case $\Gamma< p$, it reduces to the problem 
$$\min_{(x,y)}\Big\{Cx+\underline cy+\sum_{i\in I(y)}d_iy_i: \sum_{i\in N} x_i=p,  \sum_{i\in N} y_i=p, \sum_{i\in N} x_iy_i\ge p-k\Big\},$$ which we denote as
 {\sc Dis-Car}($p,k,\Gamma<p$). Below we present an $O(nk^2p)$ time dynamic programming algorithm for this problem, and show how it can be modified to solve the problem in (\ref{ProbGam}) for the case $\Gamma\ge p$ of  {\sc Dis-Car}($p,k$).

Recall that $d_1\ge \cdots\ge d_n$.  Similar to the previous dynamic programming algorithm, partial solutions $(x,y)$ are iteratively constructed. Each partial solution $(x,y)$ is associated with a state $(j,k_{x\backslash y},k_{x\cap y},k_{y\backslash x})$, where  $k_{x\backslash y}=|\{i:x_i=1,y_i=0\}|$, $k_{y\backslash x}=|\{i:x_i=0,y_i=1\}|$ and $k_{x\cap y}=|\{i:x_i=1,y_i=1\}|$. For each state $(j,k_{x\backslash y},k_{x\cap y},k_{y\backslash x})$, function $H_j(k_{x\backslash y},k_{x\cap y},k_{y\backslash x})$ is recursively calculated, which is the minimum value of $Cx+\underline cy+\sum_{i\in I(y)}d_iy_i$ among all partial solutions $(x,y)$ in this state. It is clear that if a partial solution in the state $(j,k_{x\backslash y},k_{x\cap y},k_{y\backslash x})$ can be extended to a complete optimal solution of the problem {\sc Dis-Car}($p,k,\Gamma< p$), then a solution with the value $H_j(k_{x\backslash y},k_{x\cap y},k_{y\backslash x})$ can be extended in the same way to an optimal solution as well. 

Note that the constraints $\sum_{i\in N} x_i=p$,  $\sum_{i\in N} y_i=p$ and $\sum_{i\in N} x_iy_i\ge p-k$ imply 
$k_{x\backslash y}\le k$ and $k_{y\backslash x}\le k$. The initialization is $H_0(0,0,0)=0$ and the recursion for $j\in N$, $k_{x\backslash y}=0,1,\ldots,k$,  $k_{x\cap y}=0,1,\ldots,p$, $k_{y\backslash x}=0,1,\ldots,k$, $k_{x\backslash y}+k_{x\cap y}\le p$, $k_{x\cap y}+k_{y\backslash x}\le p$, is
\begin{align*}
H_j(k_{x\backslash y},k_{x\cap y},k_{y\backslash x})=
\end{align*}
\begin{align*}
\min\begin{cases}H_{j-1}(k_{x\backslash y},k_{x\cap y},k_{y\backslash x}), {\rm\ if\ } k_{x\backslash y} + k_{x\cap y} + k_{y\backslash x} < j,                                                            & [x_j=0, y_j=0] \cr 
H_{j-1}(k_{x\backslash y}-1,k_{x\cap y},k_{y\backslash x})+C_j,                                                                         & [x_j=1, y_j=0] \cr 
H_{j-1}(k_{x\backslash y},k_{x\cap y},k_{y\backslash x}-1)+\underline c_j+d_j, {\rm\ if\ } k_{x\cap y}+k_{y\backslash x}\le \Gamma,     & [x_j=0, y_j=1] \cr 
H_{j-1}(k_{x\backslash y},k_{x\cap y},k_{y\backslash x}-1)+\underline c_j, {\rm\ if\ }  k_{x\cap y}+k_{y\backslash x}>\Gamma,   & [x_j=0, y_j=1] \cr
H_{j-1}(k_{x\backslash y},k_{x\cap y}-1,k_{y\backslash x})+C_j+\underline c_j+d_j, {\rm\ if\ } k_{x\cap y}+k_{y\backslash x}\le \Gamma, & [x_j=1, y_j=1] \cr
H_{j-1}(k_{x\backslash y},k_{x\cap y}-1,k_{y\backslash x})+C_j+\underline c_j, {\rm\ if\ } k_{x\cap y}+k_{y\backslash x}> \Gamma.   & [x_j=1, y_j=1] 
\end{cases}
\end{align*}

The optimal solution value is equal to
$$H^*=\min\{H_n(k_{x\backslash y},k_{x\cap y},k_{y\backslash x}): k_{x\backslash y}+k_{x\cap y}=p,k_{x\cap y}+k_{y\backslash x}=p,k_{y\backslash x}\le k\}.$$ 
The corresponding optimal solution $(x^*,y^*)$ can be found by backtracking. The time and space requirements of this algorithm are equal to $O(nk^2p)$.

The presented algorithm can be modified to solve the problem in (\ref{ProbGam}) for the case $\Gamma\ge p$ of  {\sc Dis-Car}($p,k$). The modification consists of removing all the  conditions ``$\ldots\le\Gamma$'' as well as the whole lines concerning the conditions ``$\ldots>\Gamma$''  under the minimum of the recursion. This modification does not change the running time estimation.

We complete this section by noting that the continuous problems {\sc Con-Car}($p$) and {\sc Con-Car}($p,k$) with the cardinality budget reduce to the same discrete problems {\sc Dis-Car}($p$) and {\sc Dis-Car}($p,k$), respectively. Indeed, if $0<\delta_i\le d_i$, then in the problems {\sc Con-Car}($p$) and {\sc Con-Car}($p,k$) it is always profitable for the adversary to set $\delta_i=d_i$ because this cost increase does not affect the cardinality constraint $|\{i\in N:\delta_i>0\}|\le\Gamma$. The algorithmic results of this section are summarized in the following theorem.
\begin{theorem} \label{TCntn} The problems {\sc Con-Car}$(p)$ and {\sc Dis-Car}$(p)$ can be solved in $O(np^2)$ time, and the problems {\sc Con-Car}$(p,k)$ and {\sc Dis-Car}$(p,k)$ can be solved in $O(nk^2p)$ time.
\end{theorem}

\section{Volume budget} \label{SecU_D1} 
In Section \ref{SSVolD}, we present NP-hardness and $\Sigma_2^p$-hardness proofs for discrete problems with volume budget and polynomial algorithms for their special cases, in which either $p$ is fixed or $\Gamma\in\{0,\infty\}$. Section~\ref{SSVolC} contains polynomial algorithms for continuous problems with volume budget.
\subsection{Discrete problems with volume budget} \label{SSVolD}
Consider the NP-complete problem {\sc Equal Cardinality Partition (ECP)}: Given positive integer numbers $a_1,\ldots,a_{2h}$ with $\sum_{i=1}^{2h}a_i=2A$, is there a subset $S\subset H=\{1,\ldots,2h\}$ such that $|S|=h$ and $\sum_{i\in S}a_i=A$? 
Assume without loss of generality that all numbers are multiples of~2. 

For given $x$ and $y$, the general problem {\sc Dis-Vol} reduces to the adversarial sub-problem that is to maximize the total cost subject to the limited budget. We denote adversarial sub-problems for the problems {\sc Dis-Vol}($p$) and {\sc Dis-Vol}($p,k)$ as {\sc Adv}$_{x,y}(p)$ and {\sc Adv}$_{x,y}(p,k)$, respectively.  

\begin{theorem} \label{PRUD1NPr2} The problems {\sc Adv}$_{x,y}(p)$ and {\sc Dis-Vol}$(p)$  are NP-hard even if $C_i=\infty$ (the fixed costs are unacceptable) and $\underline c_i=0$, $i\in N$. 
\end{theorem} 

\begin{proof}
Consider the adversarial sub-problem {\sc Adv}$_{x,y}(p)$. For any instance of {\sc ECP}, construct an instance of the decision version of {\sc Adv}$_{x,y}(p)$, in which  $N=H$, $n=p=2h$, $x=(0,\ldots,0)$, $y=(1,\ldots,1)$, $\Gamma=2Ah+A$, $C_i=\Gamma+1$, $\underline c_i=0$, $d_i=2A+a_i$, $i\in H$, and objective value $\sum_{i\in N}\delta_i\ge \Gamma$. Then, $\sum_{i\in N}\delta_i=\Gamma$ as $\sum_{i\in N}\delta_i\le \Gamma$.
It is not difficult to see that the instance of {\sc ECP} has a solution if and only if the corresponding instance of {\sc Adv}$_{x,y}(p)$ has a solution. Indeed, the necessary condition ($\Rightarrow$) is obvious. Let us prove the sufficient condition ($\Leftarrow$). Suppose that there exists a subset $S'\subseteq N$ such as $\delta_i=d_i$, $i\in S'$, and $\delta_i=0$, $i\in N\setminus S'$, which is a solution for the instance of {\sc Adv}$_{x,y}(p)$. Then, $\sum_{i\in N}\delta_i=\Gamma$ implies the following equality $2A|S'| + \sum_{i\in S'}a_i=2Ah+A$, which is equivalent to $\sum_{i\in S'}a_i= (2h+1-2|S'|)A$. The latter indicates that $|S'|$ can not be greater than $h$, as otherwise $\sum_{i\in S'}a_i$ is negative, which is impossible. It also infers that $|S'|$ can not be less than $h$, since otherwise $\sum_{i\in S'}a_i\ge 3A$, which is also impossible, since $\sum_{i\in N}a_i=2A$. Hence, $|S'|=h$ and $\sum_{i\in S'}a_i=A$, and $S'$ is also a solution for the corresponding instance of {\sc ECP}.

Note that the above proof demonstrates that the problem to calculate the objective function value of {\sc Dis-Vol}($p$) for a given solution $(x,y)$ of the decision maker is NP-hard. Let us prove that the problem to find an optimal solution for the decision maker in the problem {\sc Dis-Vol}($p$) is also NP-hard. Our approach is to demonstrate that the optimal solution of the decision maker can be limited to two solutions $(x^{(1)},y^{(1)})$ and $(x^{(2)},y^{(2)})$ with distinct total costs $F_1$ and $F_2$, respectively, such that if $F_1<F_2$, then {\sc ECP} has no solution and if $F_1>F_2$, then {\sc ECP} has a solution. Since solving {\sc Dis-Vol}($p$) determines which of these two exclusive inequalities is realized, {\sc Dis-Vol}($p$) is NP-hard.

For any instance of {\sc ECP}, construct the following instance of {\sc Dis-Vol}($p$): $N=\{1,\ldots,4h\}$, $n=4h$, $p=2h$, 
$\Gamma=2Ah+A$, $C_i=\Gamma+1$, $\underline{c}_i=0$, $d_i=2A+a_i$, $i=1,\ldots,2h$, and $C_i=\Gamma+1$, $\underline{c}_i=0$, $d_i=2Ah+A-1$, $i=2h+1,\ldots,4h$. Define  $x^{(1)}=x^{(2)}=(0,\ldots,0)$, $y^{(1)}=(1,\ldots,1,0,\ldots,0)$ and $y^{(2)}=(0,\ldots,0,1,\ldots,1)$, where there are $2h$ units and $2h$ zeros in $y^{(1)}$ and $y^{(2)}$. 
Observe that any solution $(x,y)$ other than the above two solutions has a total cost of at least $2Ah+A-1$, because either there exists at least one index $i\ge 2h+1$ for which $y_i=1$ in such a solution, and the adversary can always set $\delta_i=d_i=2Ah+A-1$, or there exists an index $i$ such that $x_i=1$, which implies $C_i\cdot x_i = \Gamma +1 = 2Ah+A+1$. Introduce set $\Delta=\{\delta: \delta_i\in \{0,d_i\},i\in N,\sum_{i\in N}\delta_i\le\Gamma\}$. 

Assume that in the instance of {\sc ECP} the answer is ``no'' (no subset $S$ exists). In this case, the minimal total costs for the solutions $(x^{(1)},y^{(1)})$ and $(x^{(2)},y^{(2)})$ are $F_1=\max\limits_{\delta\in\Delta}\sum_{i\in N}\Big(C_ix^{(1)}_i+(\underline{c}_i+\delta_i)y^{(1)}_i\Big)\le 2Ah+A-2$ and $F_2=\max\limits_{\delta\in\Delta}\sum_{i\in N}\Big(C_ix^{(2)}_i+(\underline{c}_i+\delta_i)y^{(2)}_i\Big)=2Ah+A-1$ because all numbers $a_i$ are multiples of 2. We have $F_1<F_2$ as required.
Now assume that in the instance of {\sc ECP} the answer is ``yes'' (subset $S$ exists).  In this case, $F_1=2Ah+A$ and $F_2=2Ah+A-1$ and 
$F_1>F_2$ as required. Noting that $F_1$ and $F_2$ are the smallest possible objective function values completes the proof. 
\end{proof}

Let us prove that  {\sc Dis-Vol}($p$) is $\Sigma_2^p$-hard. It is convenient to formulate a decision version of this problem, denoted as {\sc D-Dis-Vol}($p$), in the set-theoretic terminology. 

{\sc D-Dis-Vol}($p$): Given set $N=\{1,\ldots,n\}$, triples of numbers $(C_i, \underline{c}_i, d_i)$, $i\in N$, and numbers $\Gamma$, $p$ and $V$, do there exist {\it feasible} sets $X^0\subseteq N$ and $Y^0\subseteq N$ satisfying  $|X^0| + |Y^0| = p$, $X^0\cap Y^0 = \varnothing$ such that for all {\it feasible} vectors $\delta=(\delta_1,\ldots,\delta_n)$ satisfying
$\delta_i\in\{0, d_i\}$, $i\in N$, and $\sum_{i\in N}\delta_i\le\Gamma$, relation \begin{equation} \label{Sigmapp} \sum_{i\in X^0} C_i + \sum_{i\in Y^0} (\underline{c}_i + \delta_i) \le V
\end{equation} is satisfied? 


\begin{theorem}  \label{PRUD1NPrS}
  The decision problem {\sc D-Dis-Vol}$(p)$ is $\Sigma_2^p$-complete.
\end{theorem}

\begin{proof} We first note that {\sc D-Dis-Vol}$(p)$ is of the form ``{\it whether there exist feasible sets $X^0$ and $Y^0$ such that for any feasible vector $\delta$ property $P(X^0,Y^0,\delta)$ is satisfied''}. This formulation follows the terminology and notation of polynomial-time hierarchy (see, {\it e.g.}, Stockmeyer \cite{Stockmeyer1976}). Since the property in~(\ref{Sigmapp}) may be verified in polynomial time for given $X^0,Y^0$ and $\delta$, {\sc D-Dis-Vol}$(p)$ is in the class $\Sigma_2^p$. To prove $\Sigma_2^p$-completeness, consider the following $\Sigma_2^p$-complete problem: {\sc Combinatorial Interdiction (CI)} counterpart of {\sc Knapsack with Prices Equal to Weights (KPEW)} (see, {\it e.g.}, Gr\"{u}ne and  Wulf~\cite{Grune2023}), which we abbreviate as {\sc CI-KPEW}.

{\sc CI-KPEW}: Given positive integer numbers $U$, $L$, $h$, $t$, positive integer vector $(w_1,\ldots,w_h)$ and set $B\subseteq H=\{1,\ldots,h\}$, $t\le |B|$, does there exist a set $B'\subseteq B$ with $|B'|\le t$ such that for all $S\subseteq H$ relation $L\leq \sum_{i\in S} w_i \leq U$ implies $S\cap B'\neq \varnothing$? Assume without loss of generality that relation $|B'|\le t$ is replaced by $|B'|=t$ because if $S\cap B'\neq \varnothing$ for $B'\subseteq B$ then $S\cap B''\neq \varnothing$ for any $B''$ such that $B'\subset B''\subseteq B$. 

Gr\"une and  Wulf~\cite{Grune2023} study {\it subset search problems (SSPs)}. They introduce a class of {\it SSP-NP-complete} problems and prove that problem {\sc KPEW} is SSP-NP-complete (see p. 42 in \cite{Grune2023}). They also establish a class of {\sc CI} problems that are more complex counterparts of the traditional combinatorial problems, and prove that if a problem is SSP-NP-complete, then its {\sc CI} counterpart is $\Sigma_2^p$-complete (see Theorem 14 in \cite{Grune2023}). {\sc CI-KPEW} is the {\sc CI} counterpart of the SSP-NP-complete problem {\sc KPEW} according to Definition 11 in \cite{Grune2023}. It follows that {\sc CI-KPEW} is $\Sigma_2^p$-complete. 

 The main idea of our reduction of {\sc CI-KPEW} to {\sc D-Dis-Vol}$(p)$ is that the set $B'$ in {\sc CI-KPEW} corresponds to the set $N\setminus Y^0$ in {\sc D-Dis-Vol}$(p)$. 
 For any instance of {\sc CI-KPEW}, construct the following instance of {\sc D-Dis-Vol}$(p)$:  $N=H$, $n=h$, $\Gamma = U$, $V = (|B|-t)L+L-1$, $p=h-t$, $C_i=V+1$, $i\in N$, $\underline{c}_i=0$ if  
  $i\notin B$, $\underline{c}_i =L$ if $i\in B$
  and $d_i = w_i$, $i\in N$.

Consider the instance of {\sc CI-KPEW} and assume that it has a solution: there exists $B'\subseteq B$, $|B'|=t$, such that for all $S\subseteq N$ relation $L\leq \sum_{i\in S} w_i \le U$ implies $S\cap B'\neq \varnothing$. For the corresponding instance of {\sc D-Dis-Vol}$(p)$, define $X^0=\varnothing$ and $Y^0=N\setminus B'$, implying $|X^0| + |Y^0| = n-t = p$. 
  For any $\delta$ such that $\delta_i \in \{0, d_i\}$, $i\in N$, and $\sum_{i\in Y^0} \delta_i \leq \Gamma = U$, define $S(\delta)=\{i\in Y^0 :\delta_i=d_i\}$. Since $S(\delta)\subseteq Y^0$, we have $S(\delta)\cap B' = S(\delta)\cap (N\setminus Y^0) = \varnothing$. Note that relation  $L\le \sum_{i\in S(\delta)}w_i$ would contradict the assumption that {\sc CI-KPEW} has a solution, because $S(\delta)$ is one of the subsets for which  $S(\delta)\cap B'\neq \varnothing$ has to be satisfied in this case. Therefore, $\sum_{i\in S(\delta)}w_i\le L-1$ and	
  $$\sum_{i\in X^0} C_i + \sum_{i\in Y^0} (\underline{c}_i + \delta_i) = (|B| - t)L + \sum_{i\in S(\delta)} d_i \le (|B| - t)L + L - 1 = V.$$
 We deduce that the instance of {\sc D-Dis-Vol}$(p)$ has a solution.

Conversely, consider the instance of {\sc D-Dis-Vol}$(p)$ and assume that it has a solution: there exist sets $X^0\subseteq N=H$ and $Y^0\subseteq N=H$ satisfying  $|X^0|+|Y^0|=p=h-t$ and $X^0\cap Y^0 = \varnothing$ such that for all vectors $\delta=(\delta_1,\ldots,\delta_n)$ satisfying
$\delta_i\in\{0, w_i\}$, $i\in N$, and $\sum_{i\in Y^0}\delta_i\le \Gamma=U$, relation 
\begin{equation} \label{Sigmapp2}  \sum_{i\in X^0} C_i + \sum_{i\in Y^0\setminus B}\delta_i+ \sum_{i\in Y^0\cap B}(L+\delta_i)\le V=(|B|-t)L+L-1 \end{equation} holds.  Remark that $X^0=\varnothing$ because otherwise $C_i=V+1$ for some $i$ in left-hand side of inequality (\ref{Sigmapp2}) would contribute to its violation. Hence, $|Y^0|=p=h-t$. 
Besides, from $|B\cup Y^0|\le h$ and $|B\cap Y^0|=|B|+|Y^0|-|B\cup Y^0|$, it follows that $|B\cap Y^0|\ge |B|+(h-t)-h= |B|-t$. If $|B\cap Y^0|>|B|-t$, then at least $(|B|-t+1)$ number of $\underline{c}_i = L$, $i\in B$, in the left-hand side of inequality (\ref{Sigmapp2}) would contribute to its violation. Therefore, $|B\cap Y^0|=|B|-t$. 

Equalities $|Y^0|=n-t$ and $|B\cap Y^0|=|B|-t$ imply $|H\setminus Y^0|=|B\setminus Y^0|=t$. Therefore, $H=B\cup Y^0$ and $H\setminus Y^0=B\setminus Y^0$. Define $B'(Y^0)=H\setminus Y^0=B\setminus Y^0$. For any $S\subseteq H$ such that $S\cap B'(Y^0)= \varnothing$ and $\sum_{i\in S} w_i \leq U$, define vector $\delta^{(S)}$ such that $\delta^{(S)}_i=d_i=w_i$ if $i\in S$ and $\delta^{(S)}_i=0$ if $i\not\in S$. Since $(X^0,Y^0)$ is a solution of {\sc D-Dis-Vol}$(p)$, the following has to be satisfied for $\delta^{(S)}$:
  $$\sum_{i\in X^0} C_i + \sum_{i\in Y^0} (\underline{c}_i + \delta^{(S)}_i)=(|B|-t)L+\sum_{i\in S}w_i\le V=(|B|-t)L+L-1,$$
  which implies $\sum_{i\in S} w_i<L$. 
    
    In other words, for set $B'(Y^0)$, we have $|B'(Y^0)|=t$ and the following expression holds : 
    $$\forall S\subseteq H\ \ \left(S\cap B'(Y^0)=\varnothing\ \ {\rm and}\ \ \sum_{i\in S} w_i\le U\ \Rightarrow\ \ \sum_{i\in S} w_i<L \right).$$
    The latter implies
    $$\forall S\subseteq H\ \ \left(L \le \sum_{i\in S} w_i\le U\ \Rightarrow\ \ S\cap B'(Y^0)\neq\varnothing \right),$$
    which means that $B'(Y^0)$ is a solution for {\sc CI-KPEW}.
\end{proof}

Consider a problem differing from {\sc Dis-Vol}($p$) in that all $n$ costs are uncertain and $p$ items have to be selected, which we denote as {\sc All}($p$): $$\min_{S\subseteq N}\max\limits_{\delta_i\in \{0,d_i\},i\in S}\Big\{\sum_{i\in S}(\underline{c}_i+\delta_i): \sum_{i\in S}\delta_i\le\Gamma\Big\}.$$ In this problem, a solution of the decision maker can be represented as a 0-1 vector $y$ such that $y_i=1$ if and only if item $i\in N$ is selected. Denote the adversarial sub-problem of this problem for a fixed $y$ as {\sc Adv-All}$_y(p)$.

\begin{corollary} \label{CorAll} The problems {\sc Adv-All}$_y(p)$ and {\sc All}($p$) are NP-hard and the problem {\sc All}($p$) is $\Sigma_2^p$-hard  even if $\underline{c}_i=0$, $i\in N$.  
\end{corollary} 

We now pass to the hardness of the problem {\sc Dis-Vol}($p,k$).

\begin{theorem} \label{PRUD1NPrr} The problems {\sc Adv-All}$_{x,y}(p,k)$ and {\sc Dis-Vol}($p,k$) are NP-hard and the problem {\sc Dis-Vol}($p,k$) is $\Sigma_2^p$-hard even if $C_i=\underline c_i=0$, $i\in N$.  
\end{theorem} 

\begin{proof} If $C_i=0$, $i\in N$, then the $x$-part of the decision does not contribute to the total cost, and the problems {\sc Adv}$_{x,y}(p,k)$ and  {\sc Dis-Vol}($p,k$) become equivalent to {\sc Adv-All}$_y(p)$ and {\sc All}($p$). In this case, Corollary \ref{CorAll} applies for them.   
\end{proof}

\subsubsection{Polynomial special cases.} 

 Following Graham et al. \cite{Graham1994}, note that the number of subsets of the set $N$ including exactly (resp. at most) $a$ elements is equal to $\binom{n}{a}\in O(n^a)$ (resp. $\sum_{i=0}^a\binom{n}{i}\in O(n^a)$) and the number of partitions of a set with $a$ elements into $r$, $r\le a$, disjoint subsets corresponds to the Stirling number of the second kind, noted as $a \brace r$, which is in $O(r^a)$.

Therefore, if $p$ is fixed, then the problem {\sc Dis-Vol}($p$) can be solved in $O(n^p)$ time in the following way. First, we need to enumerate all the subsets of cardinality $p$ of the set $N$, which can be done in $O(n^p)$ time. Then, for each of this subset, enumerate all its partitions into two disjoint subsets. The latter can be done in $O(2^p)$ time. Now, let $\mathcal{S}$ denote a collection of all partitions into two disjoint subsets for all subsets of cardinality $p$. It is not difficult to see that each partition from $\mathcal{S}$ can be viewed as a potential solution $(x,y)\in X\times Y$ for the problem {\sc Dis-Vol}($p$), where $x_i=1$ (resp. $y_i=1$) if and only if $i$ belongs to the first (resp. second) subset of the concerned partition. To simplify the following discussion, we assume that the collection $\mathcal{S}$ is now made up of potential solutions $(x,y)$ and not partitions. Finally, for each solution $(x,y)\in \mathcal{S}$, construct the set $\Delta^{(y)}$ of all possible vectors $\delta$ with respect to $y$ such that $\Delta^{(y)}=\{\delta: \delta_i\in\{0,d_i\} {\rm \ if \ } y_i=1 {\rm \ and \ } \delta_i=0 {\rm \ if \ } y_i=0\}$. The construction of the this set can also be done in $O(2^p)$ time at most, since $|\{i\in N : y_i=1\}|\le p$ for each $(x,y)\in\mathcal{S}$. Thus, an optimal solution $(x,y,\delta)$ of the problem {\sc Dis-Vol}($p$) can be found as follows :
\begin{equation}\label{polynomial}
\min_{(x,y)\in\mathcal{S}}\max_{\delta\in \Delta^{(y)}}\Big\{\sum_{i\in N}(C_ix_i+\delta_iy_i): \sum_{i\in N}\delta_iy_i\le\Gamma\Big\},
\end{equation}
which can be done in $O(n^p\cdot 2^p \cdot 2^p)$ time, which is equal to $O(n^p)$ time for a fixed $p$.

A similar approach can be used to solve the problem {\sc Dis-Vol}($p,k$) in $O(n^{2p})$ time. It can be done by at first constructing two families $\mathcal{X}$ and $\mathcal{Y}$ of all subsets of cardinality $p$ of the set $N$, which can be done in $O(n^p)$ time, and where each subset corresponds to a hypothetic potential solution $x$ or $y$, respectively. And then by computing in $O(n^{2p})$ time the Cartesian product $\mathcal{S}$ of these two families, just  excluding pairs $(x,y)$ such that the number of indices $i$, where $x_i=0$ and $y_i=1$, is greater than $k$. In other words, $\mathcal{S}=\{(x,y)\in \mathcal{X}\times \mathcal{Y}: \sum_{i\in N} x_iy_i\ge p-k\}$. For each pair $(x,y)\in\mathcal{S}$, the set $\Delta^{(y)}$ is constructed in $O(2^p)$ time in the same way as above. Finally, following expression \eqref{polynomial}, an optimal solution $(x,y,\delta)$ of the problem {\sc Dis-Vol}($p,k$) can be found in $O(n^{2p}\cdot 2^p)$ time, which is equal to $O(n^{2p})$ time for a fixed $p$.

Let the uncertainty budget $\Gamma$ be zero or non-restricting. Consider a special case of the problem {\sc Dis-Vol}($p$) with $\Gamma\in\{0,\infty\}$. Denote $m_i=\min\{C_i,\underline c_i+o_i\}$, where $o_i=0$ if $\Gamma=0$ and $o_i=d_i$ if $\Gamma=\infty$, $i\in N$. An optimal solution $(x,y,\delta)$ of the  problem {\sc Dis-Vol}($p$) with $\Gamma\in\{0,\infty\}$ can be found by the following $O(n)$ time algorithm: select $p$ smallest values from the set $\{m_i: i\in N\}$, and if  $m_i=C_i$ for the selected value then set $x_i=1$, else if $m_i=\underline c_i+o_i$ then set $y_i=1$ and $\delta_i=o_i$. The other components of $x$, $y$ and $\delta$ are equal to zero. 
The same algorithm for the  two-stage counterpart of the problem {\sc Dis-Vol}($p$) with $\Gamma=\infty$ is described by Kasperski and Zieli\'nski \cite{Kasperski2017} (p. 60, Theorem~4). The following statement is obvious.

\begin{observation} \label{Oextr} If $\Gamma\in\{0,\infty\}$, then problems {\sc Dis-Vol}($p$), {\sc Con-Vol}($p$), {\sc Dis-Car}($p$) and {\sc Con-Car}($p$) are equivalent, and problems {\sc Dis-Vol}($p,k$), {\sc Con-Vol}($p,k$), {\sc Dis-Car}($p,k$) and {\sc Con-Car}($p,k$) are equivalent. 
\end{observation}

The algorithmic results of this section are summarized in the following theorem.
\begin{theorem} \label{TSpec} If $p$ is fixed, then problems {\sc Dis-Vol}($p$) and {\sc Dis-Vol}($p,k$) can be solved in $O(n^p)$ and $O(n^{2p})$ time, respectively.  If $\Gamma\in\{0,\infty\}$, then problems {\sc Dis-Vol}($p$), {\sc Con-Vol}($p$), {\sc Dis-Car}($p$) and {\sc Con-Car}($p$) can be solved in $O(n)$ time, and problems {\sc Dis-Vol}($p,k$), {\sc Con-Vol}($p,k$), {\sc Dis-Car}($p,k$) and {\sc Con-Car}($p,k$) can be solved in $O(nk^2p)$ time.
\end{theorem}

\subsection{Polynomial algorithms for continuous problems with volume budget} \label{SSVolC}
Consider the general problem {\sc Con-Vol} with an arbitrary set $X\times Y$ of feasible solutions $(x,y)$. For this problem, assume that some solution $(x,y)$ is fixed and the uncertainty budget is a variable parameter $\gamma\ge 0$. We denote the corresponding cost function as $Q_{x,y}(\gamma)=Cx+\max\limits_{0\le \delta_i\le d_i, i\in N}\Big\{(\underline c+\delta)y: \sum_{i\in N}\delta_i\le\gamma\Big\}$. Further discussion is supported by Figure~\ref{FLines}.

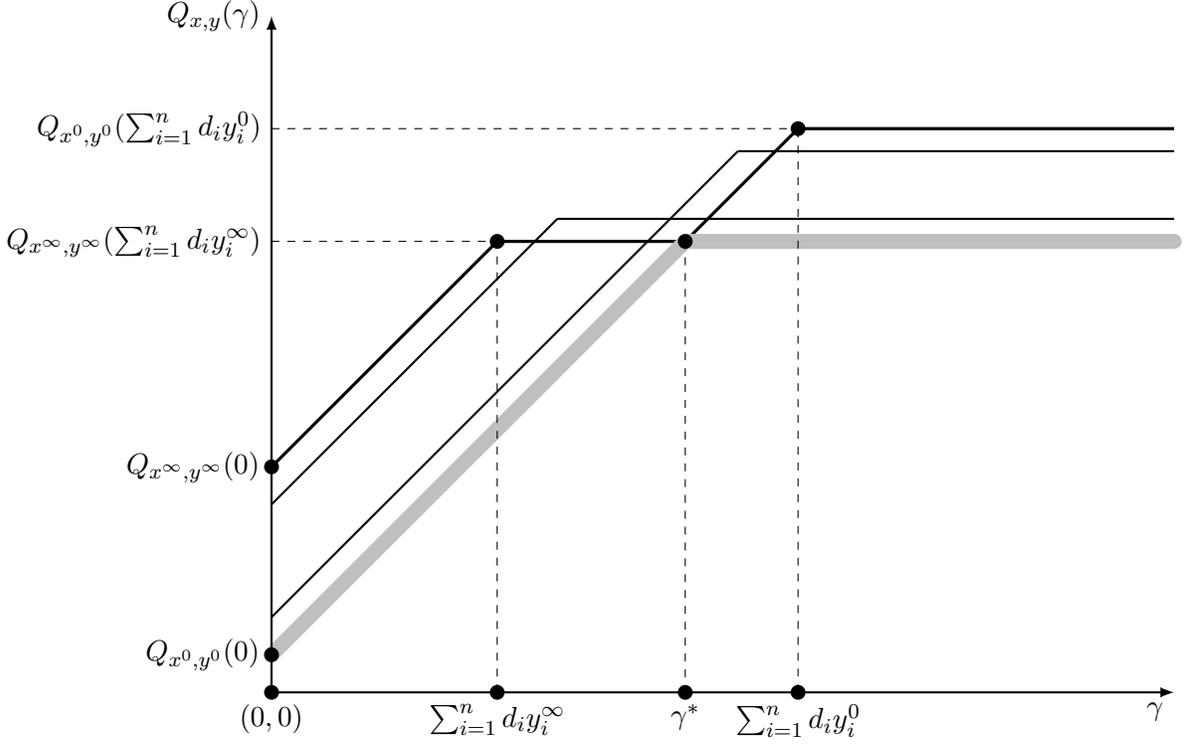
\begin{figure*}[!htb]
    \centering
\begin{tikzpicture}
 
    \draw[very thick](0,0.5)--(7,7.5);
    \draw[very thick](7,7.5)--(12,7.5);

    \draw[very thick](0,3)--(3,6);
    \draw[very thick](3,6)--(12,6);

    \draw[line width=0.2cm,color=gray!50,cap=round,join=round, opacity=0.5] (0,0.5)--(5.5,6)--(12,6);
      
    
    \draw[dashed](5.5,6)--(5.5,0);
    \draw[dashed](0,7.5)--(7,7.5);
    \draw[dashed](7,0)--(7,7.5);
    \draw[dashed](0,6)--(3,6);
    \draw[dashed](3,0)--(3,6);
    
    \draw[thick](0,2.5)--(3.8,6.3);
    \draw[thick](3.8,6.3)--(12,6.3);
		 
    \draw[thick](0,1.0)--(6.2,7.2);
    \draw[thick](6.2,7.2)--(12,7.2);

    \filldraw[ultra thick] (5.5,6) circle (2pt);
    \filldraw[ultra thick] (7,7.5) circle (2pt);
    \filldraw[ultra thick] (3,6) circle (2pt);

    \filldraw[ultra thick] (3,0) circle (2pt) node[below] {$\sum_{i=1}^n d_iy^\infty_i$};
    \filldraw[ultra thick] (0,0) circle (2pt) node[below] {$(0,0)$};
    \filldraw[ultra thick] (5.5,0) circle (2pt) node[below] {$\gamma^*$};
    \filldraw[ultra thick] (7,0) circle (2pt) node[below] {$\sum_{i=1}^n d_iy^0_i$};
		
    \filldraw[ultra thick] (0,0.5) circle (2pt) node[left] {$Q_{x^0,y^0}(0)$};
    \filldraw[ultra thick] (0,3) circle (2pt) node[left] {$Q_{x^\infty,y^\infty}(0)$};
    \node[left] at (0,7.5) {$Q_{x^0,y^0}(\sum_{i=1}^n d_iy^0_i)$};
    \node[left] at (0,6) {$Q_{x^\infty,y^\infty}(\sum_{i=1}^n d_iy^\infty_i)$};
   
    \draw[thick,->,>=latex](0,0)--(12,0) node[below left] {$\gamma$};
    \draw[thick,->,>=latex](0,0)--(0,9) node[left] {$Q_{x,y}(\gamma)$};  
\end{tikzpicture}
\caption{Functions $Q_{x,y}(\gamma)$.  Intersection point  $\gamma^*=Q_{x^\infty,y^\infty}(\sum_{i=1}^n d_iy_i^\infty)-Q_{x^0,y^0}(0)$.} \label{FLines} 
\end{figure*}

\begin{observation} \label{OLines1}For any $(x,y)\in X\times Y$, the function $Q_{x,y}(\gamma)$ is continuous piecewise linear with two linear segments for $\gamma\ge 0$. The first segment connects points $(0, Q_{x,y}(0))$ and $(\sum_{i=1}^n d_iy_i, Q_{x,y}(\sum_{i=1}^n d_iy_i))$ and has the slope coefficient 1. The second segment starts at $(\sum_{i=1}^n d_iy_i, Q_{x,y}(\sum_{i=1}^n d_iy_i))$ and is horizontal. \end{observation}  

Denote by $(x^0,y^0)$ and $(x^\infty,y^\infty)$ optimal solutions of the problem {\sc Con-Vol} with $\Gamma=0$ and $\Gamma=\infty$, respectively. In these problems, the uncertain costs are fixed to be $\underline{c}_i$, $i\in N$, and $\underline{c}_i+d_i$, $i\in N$, respectively. In the previous section (Theorem \ref{TSpec}), we have shown how to find solutions $(x^0,y^0)$ and $(x^\infty,y^\infty)$ for problems {\sc Con-Vol}($p$) and  {\sc Con-Vol}($p,k$) in $O(n)$ and $O(nk^2p)$ time, respectively. The following statement can be easily deduced from Figure~\ref{FLines}.

\begin{observation} \label{OLines2} Among all functions $Q_{x,y}(\gamma)$, $(x,y)\in X\times Y$, the lowest inclined segment is determined by the function $Q_{x^0,y^0}(\gamma)$ and the lowest horizontal segment is determined by the function $Q_{x^\infty,y^\infty}(\gamma)$. They intersect at the $\gamma$-point $\gamma^*=Q_{x^\infty,y^\infty}(\sum_{i=1}^n d_iy_i^\infty)-Q_{x^0,y^0}(0)$. 
\end{observation}  

Consider function $Q_{x^0,y^0}(\gamma)$ for $\gamma\le\gamma^*$. The corresponding $\delta^0$-vector can be determined in $O(n)$ time by the following algorithm. Calculate partial sums $\sum_{j=1}^i d_jy^0_j$ until $\sum_{j=1}^{i-1} d_jy^0_j< \gamma$ and  $\sum_{j=1}^i d_jy^0_j\ge \gamma$, $i \in N$. Set $\delta^0_j=d_jy^0_j$ for $j=1,\ldots,i-1$, $\delta^0_i= \gamma -\sum_{j=1}^{i-1} d_jy^0_j$, and $\delta^0_j=0$ for $j=i+1,i+2,\ldots,n$.
Observations \ref{OLines1} and \ref{OLines2} imply the following theorem.
\begin{theorem} \label{TCntnVol} For the general problem {\sc Con-Vol}, if solutions $(x^0,y^0)$ and $(x^\infty,y^\infty)$ can be found in $O(T)$ time, then the problem {\sc Con-Vol} can be solved in $O(n+T)$ time by the following algorithm: If $\Gamma\le \gamma^*$, then the optimal solution is $(x^0,y^0,\delta^0)$.  If $\Gamma>\gamma^*$, then the optimal solution is $(x^\infty,y^\infty,\delta^\infty)$, where $\delta^\infty_i=d_iy^\infty_i$, $i\in N$. \end{theorem}

Theorems \ref{TSpec} and \ref{TCntnVol} imply the following corollary.
\begin{corollary} \label{CCntnVol} Problems {\sc Con-Vol}$(p)$ and {\sc Con-Vol}$(p,k)$ can be solved in $O(n)$ and $O(nk^2p)$ time, respectively. \end{corollary}

\section{Weighted uncertainty} \label{SecW} The uncertainty budget considered so far in this article and the articles cited above is such that contributions of all items to it is uniform. In some practical situations, it can be useful to assume that the items contributions to the uncertainty budget are not uniform (uncertain costs of some items are more resilient to the changes than others). In this case, weights $w_i$ can be associated with items $i\in N$, and the uncertainty budget constraint can be expressed as 
$\sum\limits_{i\in N}w_i\delta_i\le\Gamma$ and $\sum\limits_{\{i\in N:\delta_i>0\}}w_i\le\Gamma$ in the case of weighted volume and weighted cardinality, respectively. A special case $w_i=1/d_i$, $i\in N$, of the ``weighted uncertainty'' has been studied by Chassein and Goerigk \cite{Chassein2021} and Brauner et al. \cite{Brauner2024} in the context of two-stage min-max-min item selection problems with alternative solutions.

We denote ``weighted'' variants of the problems studied in this paper by adding ``{\sc W-}'' in front of their notations, {\it e.g.}, {\sc W-Con-Car} and {\sc W-Dis-Vol}($p,k$).  Since problems with arbitrary weights are more general than their unit-weight counterparts and the input size is only increased by $n$ numbers, all the hardness proofs apply for the weighted counterparts. Furthermore, the difference between the ``weighted cardinality'' problem {\sc W-Dis-Car} and the ``unweighted volume'' problem {\sc Dis-Vol} is that the constraint $\sum\limits_{\{i\in N:\delta_i=d_i\}}w_i\le\Gamma$ is used instead of $\sum\limits_{\{i\in N:\delta_i=d_i\}}d_i\le\Gamma$. Therefore,  {\sc Dis-Vol} is a special case of {\sc W-Dis-Car} when $w_i=d_i$, $i\in N$, and {\sc W-Dis-Car} cannot be easier  than {\sc Dis-Vol}. Taking into account the fact that the problems {\sc W-Con-Car} and {\sc W-Dis-Car} are equivalent, we  make the following observation.

\begin{observation} \label{OWNP} If weights $w_i$, $i\in N$, are arbitrary, then problems {\sc W-Dis-Vol}($p$), {\sc W-Dis-Car}($p$) and  {\sc W-Con-Car}($p$) are NP-hard and $\Sigma^p_2$-hard, and problems {\sc W-Dis-Vol}($p,k$),  {\sc W-Dis-Car}($p,k$) and  {\sc W-Con-Car}($p,k$) are 
NP-hard. 
\end{observation}  

Consider the special case $w_i=1/d_i$, $i\in N$, in which items contributions to the uncertainty budget can be seen as 
{\it normalized}. In this case,  $\sum\limits_{\{i\in N:\delta_i=d_i\}}w_i\delta_i=|\{i\in N:\delta_i=d_i\}|$ and the problem {\sc W-Dis-Vol} reduces to 
the problem {\sc Dis-Car}. Therefore, the following observation holds.

\begin{table*}[!htbp]
\centering
\scalebox{0.85}{
\begin{tabular}{|l|l|l|}
 \hline
Problem  & Complexity & Reference\\ \hline \hline
{\sc Any-Car}($p$)  &   $O(np^2)$ & Theorem \ref{TCntn} \\ \hline
 {\sc Any-Car}($p,k$)  &    $O(nk^2p)$ & Theorem \ref{TCntn}  \\ \hline
{\sc Dis-Vol}($p$) & NP-hard & Theorem \ref{PRUD1NPr2} \\ \hline
{\sc Dis-Vol}($p$) & $\Sigma^p_2$-hard & Theorem \ref{PRUD1NPrS} \\ \hline
{\sc All}($p$) & NP-hard, $\Sigma^p_2$-hard & Corollary \ref{CorAll} \\ \hline
{\sc Dis-Vol}($p,k$) & NP-hard, $\Sigma^p_2$-hard & Theorem \ref{PRUD1NPrr} \\ \hline
 {\sc Dis-Vol}($p$) for fixed $p$  & $O(n^p)$ &  Theorem \ref{TSpec}  \\ \hline
 {\sc Dis-Vol}($p,k$) for fixed $p$ & $O(n^{2p})$ &  Theorem \ref{TSpec}  \\ \hline
{\sc Con-Vol}($p$) & $O(n)$ & Corollary \ref{CCntnVol} \\ \hline
{\sc Con-Vol}($p,k$) & $O(nk^2p)$ & Corollary \ref{CCntnVol} \\ \hline
{\sc Any-Vol}($p$),  {\sc Any-Car}($p$) for $\Gamma\in\{0,\infty\}$ &   $O(n)$ &  Theorem \ref{TSpec}  \\ \hline
{\sc Any-Vol}($p,k$), {\sc Any-Car}($p,k$)  for $\Gamma\in\{0,\infty\}$&   $O(nk^2p)$ &  Theorem \ref{TSpec}  \\ \hline
{\sc Con-Vol} & $O(n+T)$ & Theorem \ref{TCntnVol}  \\  \hline
{\sc W-Dis-Vol}($p$),{\sc W-Dis-Vol}($p,k$), {\sc W-Dis-Car}($p$), {\sc W-Dis-Car}($p,k$)  & NP-hard & Observation \ref{OWNP} \\ \hline
{\sc W-Dis-Vol}($p$) for $w_i=1/d_i$, $i\in N$ & $O(np^2)$ & Observation \ref{OWDV} \\  \hline
{\sc W-Dis-Vol}($p,k$) for $w_i=1/d_i$, $i\in N$  & $O(nk^2p)$ & Observation \ref{OWDV} \\  \hline
{\sc W-Con-Vol}($p$), {\sc W-Con-Vol}($p,k$) & Open &  \\ \hline
{\sc W-Con-Car}($p$), {\sc W-Con-Car}($p,k$) for $w_i=1/d_i$, $i\in N$ & Open &  \\ \hline
\end{tabular}
}
\caption{Computational complexity. {\sc Any}$\ \in$\{{\sc Con},{\sc Dis}\}. $T$ is time to find $(x^0,y^0)$ and $(x^\infty,y^\infty)$} \label{tab-Comp}
\end{table*}


\begin{observation} \label{OWDV} If $w_i=1/d_i$, $i\in N$, then the problems {\sc W-Dis-Vol}($p$) and {\sc W-Dis-Vol}($p,k$) can be solved in $O(np^2)$ and $O(nk^2p)$ time, respectively. 
\end{observation}  

The computational complexity of the continuous problems {\sc W-Con-Vol}($p$), {\sc W-Con-Vol}($p,k$) for arbitrary weights and the ``normalized'' problems  {\sc W-Con-Car}($p$) and {\sc W-Con-Car}($p,k$) with weights $w_i=1/d_i$, $i\in N$, remain open. We note that  ``normalized'' problems with a cardinality budget are of only theoretical interest, since the pure cardinality budget is already normalized.

\section{Conclusions} Computational complexity results of the problems studied in this paper are given in Table~\ref{tab-Comp}. For future research, it is interesting to establish the computational complexity of the open ``weighted'' problems, to develop algorithms faster than the non-linear algorithms, and to develop pseudo-polynomial time algorithms and fully polynomial time approximation schemes (FPTAS) or prove strong NP-hardness of the discrete problems with volume budget and variable $p$.
General and particular ``weighted uncertainties'' represent an interesting topic for future research of various types of uncertain decision making problems.

\bibliographystyle{elsarticle-num}
\bibliography{COMBI}



\end{document}